\documentclass[11pt]{article}
\usepackage[T1]{fontenc}
\usepackage[utf8]{inputenc}
\usepackage[english]{babel}
\usepackage{microtype}
\usepackage{amsmath,amssymb,amsfonts,amsthm,mathtools}
\usepackage[shortlabels]{enumitem}
\usepackage{geometry}
\usepackage[colorlinks=true,linkcolor=blue,citecolor=blue,urlcolor=blue]{hyperref}
\newtheorem{theorem}{Theorem}[section]
\newtheorem{lemma}[theorem]{Lemma}
\newtheorem{claim}[theorem]{Claim}

\newtheorem{problem}[theorem]{Problem}

\newcommand{\cE}{\mathcal E}

\newcommand{\cI}{\mathcal I}

\newcommand{\1}{\mathbf 1}
\newcommand{\dd}{\,d}

\usepackage{xcolor}

\usepackage{extarrows}

\newcommand{\clq}{\mathrm{cl}}
\newcommand{\st}{\mathrm{st}}
\newcommand{\Leb}{\lambda}

\title{Paths of Odd Order in Graphs with Given Edge Density}

\author{
Yuyao Yang\textsuperscript{1}\thanks{Email: \href{mailto:alaia\_y@sjtu.edu.cn}{alaia\_y@sjtu.edu.cn}}\quad
and Jiasheng Zeng\textsuperscript{2}\thanks{Email: \href{mailto:jasonzeng@mail.ustc.edu.cn}{jasonzeng@mail.ustc.edu.cn}}
}

\date{\today}
\begin{document}
\maketitle
\begin{abstract}
We determine the asymptotic maximum number of unlabelled copies of $P_{2r+1}$ in graphs with prescribed edge density, where $r\ge1$ is fixed and $P_{2r+1}$ denotes the path on $2r+1$ vertices. If an $n$ vertex graph $G$ has edge density $c=2e(G)/n^2$, then the maximum is $\frac12S_r(c)n^{2r+1}+O(n^{2r})$ for $0<c\le c_r$, and $\frac12c^{r+1/2}n^{2r+1}+O(n^{2r})$ for $c_r\le c<1$, where $S_r(c)$ is the value given by the quasi-star construction and $c_r\in(0,1)$ is an explicit algebraic transition point. Thus the quasi-star construction is asymptotically extremal below the transition, while the quasi-clique construction is asymptotically extremal above the transition. This extends the quasi-star versus quasi-clique theorem of Ahlswede and Katona for $P_3$ and the theorem of Nagy for $P_5$ to all paths with an odd number of vertices. The proof reduces the problem to threshold graphons and then to two endpoint families. The three-step endpoint is handled by reducing the required inequality to coefficient nonnegativity in a Bernstein expansion, which is proved by a direct combinatorial argument.
\end{abstract}

\section{Introduction}\label{secintro}

All graphs in this paper are finite and simple unless graphons are explicitly mentioned. If $G$ is a graph, then $V(G)$ and $E(G)$ denote its vertex set and edge set, $e(G)=|E(G)|$, and $|V(G)|$ denotes its order. For an $n$ vertex graph $G$, we use the normalized edge density $c=2e(G)/n^2$. A graph homomorphism, or simply a homomorphism, from a finite graph $H$ to a finite graph $G$ is a map $\phi:V(H)\to V(G)$ such that $\phi(u)\phi(v)\in E(G)$ whenever $uv\in E(H)$. We write $\hom(H,G)$ for the number of homomorphisms from $H$ to $G$, and $N(H,G)$ for the number of unlabelled subgraphs of $G$ isomorphic to $H$. We write $P_k$ for the path on $k$ vertices.

Problems of extremal graph counting concern the optimization of subgraph counts under prescribed density constraints. In the dense setting considered here, the prescribed parameter is the edge density. A basic minimization problem in this setting is the clique density problem. Lov\'asz and Simonovits initiated the systematic study of the minimum number of cliques forced by a given edge density \cite{LovaszSimonovits1976,LovaszSimonovits1983}. Razborov resolved the triangle case by introducing flag algebras \cite{Razborov2007}, and Reiher later proved the clique density theorem \cite{Reiher2016}. These results show that even a single edge-density constraint can lead to delicate extremal structures.

This paper concerns the corresponding maximization problem. For a fixed graph $H$ and edge density $c$, we seek the maximum possible density of $H$ subject to the edge-density constraint. Graphons provide a natural formulation of this dense extremal problem. A graphon is a symmetric measurable function $W:[0,1]^2\to[0,1]$. For a finite graph $H$, its homomorphism density in $W$ is given by
\[
t(H,W)=\int_{[0,1]^{V(H)}}\prod_{uv\in E(H)}W(x_u,x_v)\prod_{u\in V(H)}\dd x_u.
\]
In particular, $t(K_2,W)$ is the edge density of $W$. Graphon theory reduces the first-order asymptotic problem of dense graph sequences to a variational problem on graphons \cite{LovaszSzegedy2006,Lovasz2012}. For a general fixed graph $H$, the corresponding maximization problem can be written as
\[
M_H(c)=\sup\{t(H,W)\mid W\text{ is a graphon and }t(K_2,W)\le c\}.
\]

There are two most basic candidate constructions in this type of maximization problem. The first is the \textit{quasi-clique}. Given $0\le c\le1$, the quasi-clique graphon with edge density $c$ is defined as $A_{\clq}^c(x,y)=\1_{\max(x,y)<\sqrt c}.$ This means that $A_{\clq}^c$ takes value $1$ on the square $[0,\sqrt c)^2$, which corresponds to a clique-like vertex block of measure $\sqrt c$, while the remaining vertices have no edges. Clearly, 
$t(K_2,A_{\clq}^c)=c.$ If $H$ is a connected graph with $|V(H)|=v$, then any contributing homomorphism must map all vertices of $H$ into this clique-like block. Consequently, $t(H,A_{\clq}^c)=c^{v/2}.$ The second candidate construction is the \textit{quasi-star}. Let $a=\sqrt{1-c}$ and $s=1-a.$ The quasi-star graphon with edge density $c$ is defined as $A_{\st}^c(x,y)=\1_{\min(x,y)<s}.$ Here $[0,s]$ is the dominating block, while $[s,1]$ is the independent block of measure $a$. Since $t(K_2,A_{\st}^c)=1-a^2=c,$ this graphon also has edge density $c$. For any fixed graph $H$, the value $t(H,A_{\st}^c)$ can be described in terms of the independent sets of $H$: vertices mapped to the independent block must form an independent set in $H$. Therefore \[t(H,A_{\st}^c)=\sum_{I\in\cI(H)}s^{|V(H)|-|I|}a^{|I|},\] where $\cI(H)$ denotes the collection of all independent sets of $H$.

Ahlswede and Katona \cite{AhlswedeKatona1978} proved the first nontrivial case, namely that the number of copies of $P_3$ is maximized by either a quasi-star or a quasi-clique given the number of vertices and edges. Nagy later proved the analogous result for $P_5$ and provided the precise asymptotic upper bound for the transition between quasi-star and quasi-clique under fixed edge density \cite{Nagy2017}. The same two-construction phenomenon is also known for stars in several forms. Kenyon, Radin, Ren, and Sadun investigated the phase space of edge density and $k$-star density from the perspective of constrained graphons \cite{KenyonRadinRenSadun2017}. Reiher and Wagner proved that the asymptotic maximum of all $k$-edge stars is attained by either a quasi-star or a quasi-clique \cite{ReiherWagner2018}. Related ordered and colored subgraph density problems were later studied by Cairncross and Mubayi \cite{CairncrossMubayi2025}. In this paper we prove that the same quasi-star versus quasi-clique phenomenon holds for every path $P_{2r+1}$ with an odd number of vertices.

We now give the two functions that determine the extremal value in our theorem. For $0\le c\le1$, let $a=\sqrt{1-c}$ and $s=1-a$. Since the number of independent sets of size $j$ in $P_{2r+1}$ is $\binom{2r+2-j}{j}$, the $P_{2r+1}$ density given by the quasi-star graphon is
\[
S_r(c)=\sum_{j=0}^{r+1}\binom{2r+2-j}{j}s^{2r+1-j}a^j.
\]
The value given by the quasi-clique graphon is $t(P_{2r+1},A_{\clq}^c)=c^{r+1/2}$. The transition point $c_r$ is given by an explicit algebraic equation. Let $A_r(w)=\sum_{j=0}^{r+1}\binom{2r+2-j}{j}w^j$ and $E_r(w)=A_r(w)^2-(1+2w)^{2r+1}$. Since $E_r(0)=0$, the quotient $E_r(w)/w$ is a polynomial. Let $w_r$ be the unique positive zero of $E_r(w)/w$, and define $c_r=\frac{1+2w_r}{(1+w_r)^2}$. The equation $S_r(c)=c^{r+1/2}$ holds in $(0,1)$ exactly at $c=c_r$. Moreover, for $0<c<c_r$ we have $S_r(c)>c^{r+1/2}$, while for $c_r<c<1$ we have $S_r(c)<c^{r+1/2}$.

\begin{theorem}\label{thmmain}
Fix an integer $r\ge1$ and let $0<c<1$. Let $G$ be an $n$ vertex graph with edge density $c=2e(G)/n^2$. Then
\[
N(P_{2r+1},G)\le
\begin{cases}
\frac12S_r(c)n^{2r+1}+O(n^{2r}),&0<c\le c_r,\\
\frac12c^{r+1/2}n^{2r+1}+O(n^{2r}),&c_r\le c<1.
\end{cases}
\]
The bounds are asymptotically sharp. For $0<c<c_r$, the first order extremal value is attained by quasi-star graph sequences. For $c_r<c<1$, the first order extremal value is attained by quasi-clique graph sequences. At $c=c_r$, both constructions attain the same first order value.
\end{theorem}

We prove the following graphon statement and then derive the finite graph theorem from it. For $0\le c\le1$, define
\[
M_r(c)=M_{P_{2r+1}}(c)=\sup\{t(P_{2r+1},W)\mid W\text{ is a graphon and }t(K_2,W)\le c\}.
\]

\begin{theorem}\label{thmgraphonform}
For every integer $r\ge1$ and every $0\le c\le1$,
\[
M_r(c)=\max\{S_r(c),c^{r+1/2}\}.
\]
Equivalently, $M_r(c)=S_r(c)$ for $0\le c\le c_r$, and $M_r(c)=c^{r+1/2}$ for $c_r\le c\le1$.
\end{theorem}

If $G$ has vertex set $\{1,\ldots,n\}$, then $W_G$ denotes the step graphon which is equal to $1$ on $[(i-1)/n,i/n)\times[(j-1)/n,j/n)$ whenever $ij\in E(G)$, and is equal to $0$ otherwise.
The finite graph statement follows from Theorem~\ref{thmgraphonform} by the standard passage from homomorphism densities to injective copies. If $W_G$ is the step graphon associated with an $n$ vertex graph $G$, then
\[
t(P_{2r+1},W_G)=\frac{\hom(P_{2r+1},G)}{n^{2r+1}}.
\]
The number of homomorphisms from $P_{2r+1}$ to $G$ that identify two vertices of $P_{2r+1}$ is $O(n^{2r})$, and every unlabelled simple copy of $P_{2r+1}$ is counted by exactly two injective homomorphisms. Hence
\[
N(P_{2r+1},G)=\frac12t(P_{2r+1},W_G)n^{2r+1}+O(n^{2r}).
\]
The lower bounds in Theorem~\ref{thmmain} follow from graph sequences converging to $A_{\st}^c$ and $A_{\clq}^c$.

The maximization problem under a fixed edge density is closely related to Alon's theory of fixed-edge subgraph counting. Alon systematically studied the problem of maximizing the number of copies of a fixed graph $H$ given the number of edges $e$ but without fixing the number of vertices \cite{Alon1981,Alon1986}. Friedgut and Kahn extended the order-of-magnitude results for this class of problems to hypergraphs \cite{FriedgutKahn1998}. Janson, Oleszkiewicz, and Ruci\'nski also used such fixed-edge extremal estimates when studying upper tails of subgraph counts in random graphs \cite{JansonOleszkiewiczRucinski2004}. F\"uredi studied the maximum number of copies of star-forests in the fixed-edge setting \cite{Furedi1992}.  Recently, Kuang, Sun, Wang, and Zeng proved Alon's conjecture on the existence of the leading constant for the fixed-edge problem and characterized the corresponding limiting constant by a finite-core variational problem \cite{KuangSunWangZeng2026}. Very recently, Zeng developed a finite-kernel framework for sparse extremal graph counting, including sharp asymptotics attained by threshold graphons as the edge density tends to zero~\cite{Zeng2026FiniteKernel}.

However, fixed-edge problems are fundamentally different from the fixed edge density problem studied in this paper. In the fixed-edge problem, the number of vertices is not fixed, and extremal graphs are free to choose their support size. For instance, in the problem of counting $P_{2r+1}$ with only the number of edges $M$ specified, the maximum number of copies of $P_{2r+1}$ is $\Theta(M^{r+1})$~\cite{BollobasSarkar2001,BollobasSarkar2003,KuangSunWangZeng2026}, whereas in the fixed edge density problem with $c=2M/n^2$, the number of edges is of order $cn^2/2$, and the maximum number of $P_{2r+1}$ copies is $\Theta(M^{r+1/2})$, because the vertex scale itself participates in the extremal structure. Also, Gerbner, Nagy, Patk\'os, and Vizer studied the maximum number of copies of an arbitrary fixed graph $H$ given both the number of vertices and edges, and proved that when the edge density is sufficiently close to $1$, the quasi-clique is asymptotically optimal \cite{GerbnerNagyPatkosVizer2021}. This indicates that clique-like behavior at the high-density end is quite universal.

\medskip
The rest of the paper is organized as follows. Section~\ref{secprelim} contains the preliminaries. Section~\ref{secproof} gives the proof of the main result. Section~\ref{secconclusion} contains concluding remarks and further problems.

\section{Preliminaries}\label{secprelim}

Throughout this section, $\Leb$ denotes Lebesgue measure, and the same symbol is used for product Lebesgue measure when the ambient space is clear. We write $\1_E$ for the indicator function of a measurable set or event $E$. Equalities between measurable sets and measurable functions are understood up to null sets whenever this does not affect the corresponding graphon or integral. The phrase almost every always refers to Lebesgue measure.

For a measurable set $D\subseteq[0,1]$, the associated threshold graphon is $W_D(x,y)=\1_{\max(x,y)\in D}$. We use the following graphon form of the threshold reduction of Blekherman and Patel.

\begin{lemma}[Blekherman and Patel {\cite{BlekhermanPatel2024}}, graphon form]\label{lemthresholdreduction}
For every fixed graph $H$ and every $0\le c\le1$,
\[
\begin{aligned}
   &\sup\{t(H,W)\mid W\text{ is a graphon and }t(K_2,W)\le c\}\\ &\quad=\sup\left\{t(H,W_D)\mid D\subseteq[0,1]\text{ is measurable and }2\int_D t\dd t\le c\right\}. 
\end{aligned}
\]
\end{lemma}

The next lemma gives the edge density of a threshold graphon in the form used throughout the proof. 

\begin{lemma}\label{lemthresholdedge}
For every measurable set $D\subseteq[0,1]$, one has $t(K_2,W_D)=2\int_Dt\dd t$.
\end{lemma}

\begin{proof}
It suffices to prove the identity for intervals, then for finite disjoint unions of intervals, and then for measurable sets by Lebesgue regularity. If $D=[\alpha,\beta]$, then $t(K_2,W_D)=\Leb\{(x,y)\in[0,1]^2\mid \alpha\le\max(x,y)\le\beta\}=\beta^2-\alpha^2=\int_\alpha^\beta2t\dd t$. Finite disjoint unions follow by additivity. For a measurable set $D$, choose finite unions of intervals $E_k$ such that $\Leb(D\triangle E_k)\to0$. Since $\Leb\{(x,y)\mid \max(x,y)\in D\triangle E_k\}\le2\Leb(D\triangle E_k)$, both sides converge from $E_k$ to $D$.
\end{proof}

We next replace the set $D$ by its increasing quantile. This turns the optimization over measurable sets into an optimization over nondecreasing functions.

\begin{lemma}\label{lemquantile}
Let $D\subseteq[0,1]$ be measurable and let $m=\Leb(D)$. There is a nondecreasing function $q:[0,m]\to[0,1]$ such that $\int_D f(t)\dd t=\int_0^m f(q(u))\dd u$ for every bounded measurable function $f$. Moreover $q(u)=u+g(u)$ for a nonnegative nondecreasing function $g$, one has $0\le g(u)\le1-m$ for almost every $u$, and $\Leb(D\cap[0,q(u)])=u$ for almost every $u\in[0,m]$.
\end{lemma}

\begin{proof}
Define $F(t)=\Leb(D\cap[0,t])$ for $0\le t\le1$. If $0\le s\le t\le1$, then $F(t)-F(s)=\Leb(D\cap(s,t])$, and hence $0\le F(t)-F(s)\le t-s$. Therefore $F$ is nondecreasing and continuous. Also we have $F(0)=0$ and $F(1)=m$. For $u\in[0,m]$, define $q(u)=\inf\{t\in[0,1]\mid F(t)\ge u\}$. The set in this definition is nonempty because $F(1)=m\ge u$, and hence $q$ is well-defined. If $0\le u\le v\le m$, then $\{t\in[0,1]\mid F(t)\ge v\}\subseteq\{t\in[0,1]\mid F(t)\ge u\}$, so $q(u)\le q(v)$. Thus $q$ is nondecreasing.

We claim that $F(q(u))=u$ for every $u\in[0,m]$. Indeed, by the definition of infimum, for every $n\ge1$ there exists $t_n\in[0,1]$ such that $t_n<q(u)+1/n$ and $F(t_n)\ge u$. Since $F$ is continuous, this gives $F(q(u))\ge u$. If $F(q(u))>u$, then continuity gives some $s<q(u)$ with $F(s)\ge u$, unless $q(u)=0$. The case $q(u)=0$ gives $u=0$ because $F(0)=0$, so it cannot satisfy $F(q(u))>u$. Thus $F(q(u))\le u$. Hence $F(q(u))=u$.

We now prove that $\int_D f(t)\dd t=\int_0^m f(q(u))\dd u$ for every bounded measurable function $f$. For every $t\in[0,1]$, we first show that $\Leb\{u\in[0,m]\mid q(u)\le t\}=F(t)$. If $q(u)\le t$, then $u=F(q(u))\le F(t)$. Conversely, if $u\le F(t)$, then $t\in\{s\in[0,1]\mid F(s)\ge u\}$, and hence $q(u)\le t$. Thus $\{u\in[0,m]\mid q(u)\le t\}$ agrees with $[0,F(t)]$ up to endpoints, so $\Leb\{u\in[0,m]\mid q(u)\le t\}=F(t)$. For each Borel set $A\subseteq[0,1]$, define $\mu(A)=\Leb\{u\in[0,m]\mid q(u)\in A\}$ and $\nu(A)=\Leb(D\cap A)$. The preceding equality implies $\mu([0,t])=\nu([0,t])$ for every $t\in[0,1]$. Since the class of intervals $[0,t]$ determines finite Borel measures on $[0,1]$, it follows that $\mu(A)=\nu(A)$ for every Borel set $A$. After completing the two measures, the same equality holds for every Lebesgue measurable set $A$. Applying this equality first to indicator functions, then to simple functions, and finally to bounded measurable functions, we obtain $\int_D f(t)\dd t=\int_0^m f(q(u))\dd u$ for every bounded measurable function $f$.

Finally, we define $g(u)=q(u)-u$ for $0\le u\le m$ and it satisfies that
\[
g(u)=q(u)-F(q(u))=\Leb([0,q(u)]\setminus D)\leq\Leb([0,1]\setminus D)=1-m. 
\]
Since $F(t)\le t$ for every $t\in[0,1]$, the identity $F(q(u))=u$ gives $u\le q(u)$. Hence $g(u)\ge0$. If $0\le u\le v\le m$, then $q(u)\le q(v)$, and $v-u=F(q(v))-F(q(u))=\Leb(D\cap(q(u),q(v)])\le q(v)-q(u)$. Therefore $g(v)-g(u)=q(v)-q(u)-(v-u)\ge0$, so $g$ is nondecreasing. The identity $\Leb(D\cap[0,q(u)])=u$ follows from $F(q(u))=u$ and the definition of $F$. This proves the lemma.
\end{proof}

The edge constraint becomes a simple integral constraint in the quantile coordinates.

\begin{lemma}\label{lemedgeconstraint}
With the notation of Lemma~\ref{lemquantile}, the edge density of $W_D$ is $m^2+2G$, where $G=\int_0^mg(u)\dd u$. Hence $2\int_Dt\dd t=c$ is equivalent to $G=(c-m^2)/2$.
\end{lemma}

\begin{proof}
By Lemmas~\ref{lemthresholdedge} and~\ref{lemquantile}, one has \[t(K_2,W_D)=2\int_Dt\dd t=2\int_0^mq(u)\dd u=2\int_0^m(u+g(u))\dd u=m^2+2G.\]
\end{proof}

The following exact expansion is the point where independent sets of the path enter the proof. Vertices sent outside $D$ must form an independent set, and each such vertex contributes a factor determined by the minimum of the neighbouring quantile variables. For a vertex $i$ of the path, let $N(i)$ denote its set of neighbours in the path. For $I\subseteq V(P_{2r+1})$, let $N(I)=\{v\in V(P_{2r+1})\setminus I\mid v\text{ has a neighbour in }I\}$ denote the open neighbourhood of $I$. 

\begin{lemma}\label{lemexactexpansion}
Let $D\subseteq[0,1]$ be measurable and write $q(u)=u+g(u)$ as in Lemma~\ref{lemquantile}. Then
\[
t(P_{2r+1},W_D)=\sum_{I\in\cI(P_{2r+1})}\int_{[0,m]^{V(P_{2r+1})\setminus I}}\prod_{i\in I}g\left(\min_{j\in N(i)}u_j\right)\prod_{j\in V(P_{2r+1})\setminus I}\dd u_j.
\]
\end{lemma}

\begin{proof}
Let $V=V(P_{2r+1})=\{0,1,\ldots,2r\}$ and let $E=E(P_{2r+1})$. Write $D^c=[0,1]\setminus D$. For each $I\subseteq V$, let $J_I=V\setminus I$ and define
\[
R_I=\{(x_v)_{v\in V}\in[0,1]^V\mid x_i\in D^c\text{ for }i\in I,\ x_j\in D\text{ for }j\in J_I\}.
\]
The sets $R_I$ partition $[0,1]^V$, and therefore
\[
t(P_{2r+1},W_D)=\sum_{I\subseteq V}\int_{R_I}\prod_{ab\in E}W_D(x_a,x_b)\prod_{v\in V}\dd x_v.
\]
If $I$ is not independent, then there is an edge $ab\in E$ with $a,b\in I$. On $R_I$ one has $x_a,x_b\in D^c$, so $\max(x_a,x_b)\in D^c$ and hence $W_D(x_a,x_b)=0$. Thus only independent sets $I$ can have a nonzero term, and
\[
t(P_{2r+1},W_D)=\sum_{I\in\cI(P_{2r+1})}\int_{R_I}\prod_{ab\in E}W_D(x_a,x_b)\prod_{v\in V}\dd x_v.
\]

Fix $I\in\cI(P_{2r+1})$. On $R_I$, if $ab\in E$ and $a,b\in J_I$, then $x_a,x_b\in D$, so $\max(x_a,x_b)\in D$ and $W_D(x_a,x_b)=1$. If $i\in I$ and $j\in J_I$ with $ij\in E$, then $x_i\in D^c$ and $x_j\in D$, and
$
W_D(x_i,x_j)=\1_{\{x_i\le x_j\}}.
$
Indeed, if $x_i\le x_j$, then $\max(x_i,x_j)=x_j\in D$; if $x_i>x_j$, then $\max(x_i,x_j)=x_i\in D^c$. Hence
\[
\1_{R_I}(x)\prod_{ab\in E}W_D(x_a,x_b)=\1_{R_I}(x)\prod_{i\in I}\prod_{j\in N(i)}\1_{\{x_i\le x_j\}}.
\]
Therefore the term indexed by $I$ is
\[
\mathcal T_I=\int_{D^{J_I}}\int_{(D^c)^I}\prod_{i\in I}\prod_{j\in N(i)}\1_{\{x_i\le x_j\}}\prod_{i\in I}\dd x_i\prod_{j\in J_I}\dd x_j.
\]
For fixed $(x_j)_{j\in J_I}\in D^{J_I}$, the variables $(x_i)_{i\in I}$ occur separately. Thus
\[
\begin{aligned}
\int_{(D^c)^I}\prod_{i\in I}\prod_{j\in N(i)}\1_{\{x_i\le x_j\}}\prod_{i\in I}\dd x_i
&=\prod_{i\in I}\int_{D^c}\prod_{j\in N(i)}\1_{\{x_i\le x_j\}}\dd x_i\\
&=\prod_{i\in I}\Leb\left(D^c\cap\left[0,\min_{j\in N(i)}x_j\right]\right).
\end{aligned}
\]
Consequently, we obtain that
\[
\mathcal T_I=\int_{D^{J_I}}\prod_{i\in I}\Leb\left(D^c\cap\left[0,\min_{j\in N(i)}x_j\right]\right)\prod_{j\in J_I}\dd x_j.
\]

We now replace the variables in $D$ by quantile variables. Applying Lemma~\ref{lemquantile} successively to the variables indexed by $J_I$ yields
\[
\mathcal T_I=\int_{[0,m]^{J_I}}\prod_{i\in I}\Leb\left(D^c\cap\left[0,\min_{j\in N(i)}q(u_j)\right]\right)\prod_{j\in J_I}\dd u_j.
\]
Since $q$ is nondecreasing, for each $i\in I$ one has
$
\min_{j\in N(i)}q(u_j)=q\left(\min_{j\in N(i)}u_j\right).
$
Let $\sigma_i=\min_{j\in N(i)}u_j$. By Lemma~\ref{lemquantile}, $\Leb(D\cap[0,q(\sigma_i)])=\sigma_i$ and $q(\sigma_i)=\sigma_i+g(\sigma_i)$. Hence
\[
\Leb(D^c\cap[0,q(\sigma_i)])=\Leb([0,q(\sigma_i)])-\Leb(D\cap[0,q(\sigma_i)])=q(\sigma_i)-\sigma_i=g(\sigma_i).
\]
Therefore
\[
\mathcal T_I=\int_{[0,m]^{J_I}}\prod_{i\in I}g\left(\min_{j\in N(i)}u_j\right)\prod_{j\in J_I}\dd u_j.
\]
Summing this identity over all $I\in\cI(P_{2r+1})$ we obtain that 
\[
t(P_{2r+1},W_D)=\sum_{I\in\cI(P_{2r+1})}\int_{[0,m]^{V\setminus I}}\prod_{i\in I}g\left(\min_{j\in N(i)}u_j\right)\prod_{j\in V\setminus I}\dd u_j,
\]
and this proves the lemma.
\end{proof}

The next estimate gives the possible range of the one scalar parameter that measures how far the nondecreasing function $g$ is from being constant. 

\begin{lemma}\label{lemdefectrange}
Let $h>0$ and $m>0$. Let $0\le g\le h$ be nondecreasing on $[0,m]$. Let $G=\int_0^mg(u)\dd u$ and let $J=\int_0^m(u-m/2)g(u)\dd u$. Then
\[
0\le J\le \frac12\left(mG-\frac{G^2}{h}\right).
\]
The lower endpoint is attained by $g\equiv G/m$. The upper endpoint is attained by $g=h\1_{[m-G/h,m]}$.
\end{lemma}

\begin{proof}
For $0\le t\le h$, let $L(t)=\Leb\{u\in[0,m]\mid g(u)\ge t\}$. Since $g$ is nondecreasing, the set $\{u\in[0,m]\mid g(u)\ge t\}$ agrees up to null sets with the right interval $[m-L(t),m]$. Since $0\le g\le h$, one has $g(u)=\int_0^h\1_{\{t\le g(u)\}}\dd t$. Hence, by Fubini's theorem,
\[
G=\int_0^m g(u)\dd u=\int_0^m\int_0^h\1_{\{t\le g(u)\}}\dd t\dd u=\int_0^h\int_0^m\1_{\{g(u)\ge t\}}\dd u\dd t=\int_0^h L(t)\dd t.
\]
Similarly,
\[
\begin{aligned}
J
&=\int_0^m\left(u-\frac{m}{2}\right)g(u)\dd u=\int_0^h\int_0^m\left(u-\frac{m}{2}\right)\1_{\{g(u)\ge t\}}\dd u\dd t\\
&=\int_0^h\int_{m-L(t)}^m\left(u-\frac{m}{2}\right)\dd u\dd t=\frac{1}{2}\int_0^h L(t)(m-L(t))\dd t.
\end{aligned}
\]
Since $0\le L(t)\le m$, the last expression implies $J\ge0$. Let $\phi(x)=x(m-x)$ for $0\le x\le m$. Then $\phi$ is concave. By Jensen's inequality,
\[
J=\frac{1}{2}\int_0^h\phi(L(t))\dd t\le \frac{h}{2}\,\phi\left(\frac{1}{h}\int_0^hL(t)\dd t\right)=\frac{1}{2}\left(mG-\frac{G^2}{h}\right).
\]
It remains to check the two stated equality cases. Since $0\le G\le mh$, the function $g\equiv G/m$ satisfies $0\le g\le h$ and is nondecreasing. For this function,
\[
J=\int_0^m\left(u-\frac{m}{2}\right)\frac{G}{m}\dd u=\frac{G}{m}\left[\frac{1}{2}u^2-\frac{m}{2}u\right]_0^m=0.
\]
The function $g=h\1_{[m-G/h,m]}$ also satisfies $0\le g\le h$ and is nondecreasing. For this function,
\[
J=h\int_{m-G/h}^m\left(u-\frac{m}{2}\right)\dd u=h\left[\frac{1}{2}u^2-\frac{m}{2}u\right]_{m-G/h}^m=\frac{1}{2}\left(mG-\frac{G^2}{h}\right).
\]
This proves the lemma.
\end{proof}

We shall need one moment estimate for the unique largest independent set of the odd path.

\begin{lemma}\label{lemmoment}
Let $h>0$ and $m>0$. Let $0\le g\le h$ be nondecreasing on $[0,m]$. Let $G=\int_0^mg(u)\dd u$ and let $J=\int_0^m(u-m/2)g(u)\dd u$. Then
\[
\left(\int_0^m g(u)^{3/2}\dd u\right)^2\le \frac{G^3}{m}+\frac{2hG}{m}J.
\]
\end{lemma}

\begin{proof}
For $0\le s\le1$, define $y(s)$ by $g(ms)=hy(s)$. Then $0\le y\le1$ and $y$ is nondecreasing on $[0,1]$. Let $\mu=\int_0^1y(s)\dd s$ and $\eta=\int_0^1(s-\frac{1}{2})y(s)\dd s$. We have
\[
G=\int_0^m g(u)\dd u=m\int_0^1 g(ms)\dd s=mh\int_0^1y(s)\dd s=mh\mu,
\]
and
\[
J=\int_0^m\left(u-\frac{m}{2}\right)g(u)\dd u=m^2h\int_0^1\left(s-\frac{1}{2}\right)y(s)\dd s=m^2h\eta.
\]
Moreover, $
\left(\int_0^m g(u)^{3/2}\dd u\right)^2=m^2h^3\left(\int_0^1y(s)^{3/2}\dd s\right)^2$ and 
therefore, after substituting these three identities and dividing by $m^2h^3$, it suffices to show that
$
\left(\int_0^1y(s)^{3/2}\dd s\right)^2\le \mu^3+2\mu\eta.
$ 
By Cauchy's inequality,
\[
\left(\int_0^1y(s)^{3/2}\dd s\right)^2\le \left(\int_0^1y(s)\dd s\right)\left(\int_0^1y(s)^2\dd s\right)=\mu\int_0^1y(s)^2\dd s.
\]
It remains to prove
$
\int_0^1y(s)^2\dd s\le \mu^2+2\eta.
$
For $0\le t\le1$, let $L(t)=\Leb\{s\in[0,1]\mid y(s)\ge t\}$. Since $y$ is nondecreasing, the set $\{s\in[0,1]\mid y(s)\ge t\}$ agrees up to null sets with the right interval $[1-L(t),1]$. Since $0\le y\le1$, by Fubini's theorem,
\[
\mu=\int_0^1y(s)\dd s=\int_0^1\int_0^1\1_{\{t\le y(s)\}}\dd t\dd s=\int_0^1L(t)\dd t.
\]
Similarly,
\[
\begin{aligned}
\eta
&=\int_0^1\left(s-\frac{1}{2}\right)y(s)\dd s=\int_0^1\int_0^1\left(s-\frac{1}{2}\right)\1_{\{t\le y(s)\}}\dd t\dd s\\
&=\int_0^1\int_{1-L(t)}^1\left(s-\frac{1}{2}\right)\dd s\dd t=\frac{1}{2}\int_0^1L(t)(1-L(t))\dd t.
\end{aligned}
\]
Also,
\[
\begin{aligned}
\int_0^1y(s)^2\dd s
&=\int_0^1\int_0^1\int_0^1\1_{\{t\le y(s)\}}\1_{\{u\le y(s)\}}\dd t\dd u\dd s\\
&=\int_0^1\int_0^1\Leb\{s\in[0,1]\mid y(s)\ge t\text{ and }y(s)\ge u\}\dd t\dd u\\
&=\int_0^1\int_0^1\min\{L(t),L(u)\}\dd t\dd u.
\end{aligned}
\]
Since $\mu=\int_0^1L(t)\dd t$, we have that 
$\mu^2=\int_0^1\int_0^1L(t)L(u)\dd t\dd u.$
For $X,Y\in[0,1]$, one has
\[
\min\{X,Y\}-XY\le \frac{1}{2}X(1-X)+\frac{1}{2}Y(1-Y).
\]
Applying this inequality with $X=L(t)$ and $Y=L(u)$, we obtain
\[
\begin{aligned}
\int_0^1y(s)^2\dd s-\mu^2
&=\int_0^1\int_0^1\left(\min\{L(t),L(u)\}-L(t)L(u)\right)\dd t\dd u\\
&\le \int_0^1\int_0^1\left(\frac{1}{2}L(t)(1-L(t))+\frac{1}{2}L(u)(1-L(u))\right)\dd t\dd u\\
&=\int_0^1L(t)(1-L(t))\dd t=2\eta.
\end{aligned}
\]
Thus $\int_0^1y(s)^2\dd s\le \mu^2+2\eta$. Hence $\left(\int_0^1y(s)^{3/2}\dd s\right)^2\le \mu^3+2\mu\eta$. Scaling back proves the lemma.
\end{proof}

We now describe the two endpoint graphons that arise from Lemma~\ref{lemdefectrange}. Let $0<m<1$ and let $0<G<m(1-m)$. Let $D_2=[G/m,G/m+m]$ and $D_3=[0,m-G/(1-m)]\cup[1-G/(1-m),1]$. Both threshold graphons $W_{D_2}$ and $W_{D_3}$ have edge density $m^2+2G$. We write $T_r^{(2)}(m,G)=t(P_{2r+1},W_{D_2})$ and $T_r^{(3)}(m,G)=t(P_{2r+1},W_{D_3})$.

The following theorem is the collapse step. It says that after fixing $m$ and $G$, no threshold graphon is better than one of the two endpoint choices above.

\begin{theorem}\label{thmonedefectcollapse}
Let $r\ge2$ and let $W_D$ be a threshold graphon. Let $m=\Leb(D)$, let $q(u)=u+g(u)$ be as in Lemma~\ref{lemquantile}, and let $G=\int_0^mg(u)\dd u$. If $0<m<1$ and $0<G<m(1-m)$, then
\[
t(P_{2r+1},W_D)\le \max\{T_r^{(2)}(m,G),T_r^{(3)}(m,G)\}.
\]
The boundary cases follow by continuity.
\end{theorem}

\begin{proof}
Let $J=\int_0^m(u-\frac{m}{2})g(u)\dd u$ and let $B=mG-2J$. Since $0\le g\le1-m$, Lemma~\ref{lemdefectrange} applied with $h=1-m$ yields $0\le J\le J_{\max}$, where
$
J_{\max}=\frac{1}{2}\left(mG-\frac{G^2}{1-m}\right).
$
Moreover, notice that 
\[
B=mG-2J=m\int_0^m g(u)\dd u-2\int_0^m\left(u-\frac{m}{2}\right)g(u)\dd u=2\int_0^m(m-u)g(u)\dd u.
\]
Hence
\[
\begin{aligned}
\int_{[0,m]^2}g(\min(u,v))\dd u\dd v
&=\int_0^m\int_0^u g(v)\dd v\dd u+\int_0^m\int_u^m g(u)\dd v\dd u\\
&=\int_0^m(m-v)g(v)\dd v+\int_0^m(m-u)g(u)\dd u\\
&=2\int_0^m(m-u)g(u)\dd u=B.
\end{aligned}
\]
Let $V=V(P_{2r+1})=\{0,1,\ldots,2r\}$ and let $I_0=\{0,2,\ldots,2r\}$ be the unique independent set of size $r+1$ in $P_{2r+1}$. By Lemma~\ref{lemexactexpansion}, for each $I\in\cI(P_{2r+1})$ let
\[
\mathcal T_I=\int_{[0,m]^{V\setminus I}}\prod_{i\in I}g\left(\min_{j\in N(i)}u_j\right)\prod_{j\in V\setminus I}\dd u_j.
\]
Then
\[
t(P_{2r+1},W_D)=\sum_{I\in\cI(P_{2r+1})}\mathcal T_I.
\]

We first consider $I\in\cI(P_{2r+1})$ with $I\ne I_0$. Decompose $I$ into maximal chains under distance-two adjacency along the path. Thus each chain has the form $\{v,v+2,\ldots,v+2(t-1)\}$. If such a chain meets neither endpoint $0$ nor endpoint $2r$, then it has $t+1$ neighbours. If such a chain meets exactly one endpoint, then it has $t$ neighbours. Since $I\ne I_0$, no chain meets both endpoints. Hence $\nu(I)=|N(I)|-|I|$ is the number of chains which meet neither endpoint.

Choose one vertex from each chain which meets neither endpoint, and let $R$ be the set of these chosen vertices. Then $|R|=\nu(I)$. The two-neighbour sets $N(i)$ for $i\in R$ are pairwise disjoint. We next choose an injective map $\phi:I\setminus R\to N(I)\setminus\bigcup_{i\in R}N(i)$ with $\phi(i)\in N(i)$ for every $i\in I\setminus R$. This is done chain by chain. If a chain meets the left endpoint, match every vertex in the chain to its right neighbour. If a chain meets the right endpoint, match every vertex in the chain to its left neighbour. If a chain $\{v,v+2,\ldots,v+2(t-1)\}$ meets neither endpoint and the chosen vertex is $v+2k$, remove the two neighbours $v+2k-1$ and $v+2k+1$, match the vertices to the left of $v+2k$ to their left neighbours, and match the vertices to the right of $v+2k$ to their right neighbours. The chainwise matchings use precisely the remaining neighbours of the chains, and different chains have disjoint neighbourhoods.

For $i\in I\setminus R$, by the monotonicity of $g$, we have 
$
g\left(\min_{j\in N(i)}u_j\right)\le g(u_{\phi(i)}).
$
Therefore
\[
\begin{aligned}
\mathcal T_I
&=\int_{[0,m]^{V\setminus I}}\prod_{i\in R}g\left(\min_{j\in N(i)}u_j\right)\prod_{i\in I\setminus R}g\left(\min_{j\in N(i)}u_j\right)\prod_{j\in V\setminus I}\dd u_j\\
&\le \int_{[0,m]^{V\setminus I}}\prod_{i\in R}g\left(\min_{j\in N(i)}u_j\right)\prod_{i\in I\setminus R}g(u_{\phi(i)})\prod_{j\in V\setminus I}\dd u_j.
\end{aligned}
\]
The variables used in the two-neighbour factors indexed by $R$ are disjoint from the variables used in the one-neighbour factors indexed by $I\setminus R$. Hence the integral factors as
\[
\begin{aligned}
\mathcal T_I
&\le \prod_{i\in R}\int_{[0,m]^2}g(\min(u,v))\dd u\dd v\prod_{i\in I\setminus R}\int_0^m g(u)\dd u\cdot m^{|V|-|I|-|N(I)|}\\
&=B^{|R|}G^{|I|-|R|}m^{2r+1-|I|-|N(I)|}=B^{\nu(I)}G^{2|I|-|N(I)|}m^{2r+1-|I|-|N(I)|}.
\end{aligned}
\]

It remains to estimate $\mathcal T_{I_0}$. By Lemma~\ref{lemexactexpansion},
\[
\mathcal T_{I_0}=\int_{[0,m]^r}g(u_1)g(u_r)\prod_{j=1}^{r-1}g(\min(u_j,u_{j+1}))\dd u_1\cdots\dd u_r.
\]
Since $g$ is nondecreasing, one has
$
g(\min(u_j,u_{j+1}))=\min\{g(u_j),g(u_{j+1})\}\le\sqrt{g(u_j)g(u_{j+1})}.
$
Hence
\[
\mathcal T_{I_0}
\le \int_{[0,m]^r}g(u_1)^{3/2}g(u_r)^{3/2}\prod_{j=2}^{r-1}g(u_j)\dd u_1\cdots\dd u_r=\left(\int_0^m g(u)^{3/2}\dd u\right)^2G^{r-2}.
\]
Applying Lemma~\ref{lemmoment} with $h=1-m$, we obtain
$
\mathcal T_{I_0}\le \frac{G^{r+1}}{m}+\frac{2(1-m)G^{r-1}}{m}J.
$
Combining the estimates of $I\neq I_0$, we have
$
t(P_{2r+1},W_D)\le\Psi_{r,m,G}(J),
$
where
\[
\Psi_{r,m,G}(J)=\sum_{\substack{I\in\cI(P_{2r+1})\\ I\ne I_0}}(mG-2J)^{\nu(I)}G^{2|I|-|N(I)|}m^{2r+1-|I|-|N(I)|}+\frac{G^{r+1}}{m}+\frac{2(1-m)G^{r-1}}{m}J.
\]
For $0\le J\le J_{\max}$, $
mG-2J\ge mG-2J_{\max}=\frac{G^2}{1-m}>0.$
Each term $(mG-2J)^q$ with $q\ge0$ is constant, linear, or convex on $[0,J_{\max}]$, and the final term is linear. Hence $\Psi_{r,m,G}$ is convex on $[0,J_{\max}]$. Therefore,
\[
\Psi_{r,m,G}(J)\le\max\{\Psi_{r,m,G}(0),\Psi_{r,m,G}(J_{\max})\}.
\]

Next, we determine the values of $\Psi_{r,m,G}(0)$ and $\Psi_{r,m,G}(J_{\max})$.

\begin{claim}\label{D_2}
$\Psi_{r,m,G}(0)=T_r^{(2)}(m,G)$.
\end{claim}

\begin{proof}[Proof of Claim~\ref{D_2}]
Let $D_2=[G/m,G/m+m]$. Since $0<G<m(1-m)$, this interval is contained in $[0,1]$. The corresponding quantile obtained from Lemma~\ref{lemquantile} is $q_2(u)=u+G/m$ for $0\le u\le m$, and hence $g_2(u)=G/m$. By Lemma~\ref{lemexactexpansion},
\[
\begin{aligned}
T_r^{(2)}(m,G)
&=t(P_{2r+1},W_{D_2})=\sum_{I\in\cI(P_{2r+1})}\int_{[0,m]^{V\setminus I}}\prod_{i\in I}\frac{G}{m}\prod_{j\in V\setminus I}\dd u_j\\
&=\sum_{I\in\cI(P_{2r+1})}\left(\frac{G}{m}\right)^{|I|}m^{2r+1-|I|}=\sum_{I\in\cI(P_{2r+1})}G^{|I|}m^{2r+1-2|I|}\\ &=\Psi_{r,m,G}(0).
\end{aligned}
\]
\end{proof}

\begin{claim}\label{D_3}
$\Psi_{r,m,G}(J_{\max})=T_r^{(3)}(m,G)$.
\end{claim}

\begin{proof}[Proof of Claim~\ref{D_3}]
Let $\rho=G/(1-m)$ and let $Q=[m-\rho,m]$. Since $0<G<m(1-m)$, one has $0<\rho<m$. Let $D_3=[0,m-\rho]\cup[1-\rho,1]$. The corresponding quantile obtained from Lemma~\ref{lemquantile} satisfies $g_3(u)=(1-m)\1_Q(u)$. Moreover, we have $mG-2J_{\max}=\frac{G^2}{1-m}$. For $I\in\cI(P_{2r+1})$, let $\mathcal T_I^{(3)}$ denote the term indexed by $I$ in the expansion of $t(P_{2r+1},W_{D_3})$ from Lemma~\ref{lemexactexpansion}.

We first compute $\mathcal T_I^{(3)}$ for $I\ne I_0$. Fix such an independent set $I$ and decompose it into maximal distance-two chains. Consider a chain which meets neither endpoint and contains $t$ vertices of $I$. Its $t+1$ neighbour variables will be denoted by $a_0,\ldots,a_t$. For this chain,
\[
\prod_{k=0}^{t-1}g_3(\min(a_k,a_{k+1}))=(1-m)^t\prod_{k=0}^t\1_Q(a_k).
\]
Hence the integral over these $t+1$ neighbour variables is $(1-m)^t\rho^{t+1}=\frac{G^2}{1-m}G^{t-1}$. Now consider a chain which meets exactly one endpoint and contains $t$ vertices of $I$. It has $t$ neighbour variables, denoted by $a_1,\ldots,a_t$. For this chain, the product is nonzero exactly when all these variables lie in $Q$, and the integral over these $t$ variables is $(1-m)^t\rho^t=G^t$. Multiplying over all chains of $I$ and integrating the variables outside $N(I)$, we obtain
\[
\mathcal T_I^{(3)}=\left(\frac{G^2}{1-m}\right)^{\nu(I)}G^{2|I|-|N(I)|}m^{2r+1-|I|-|N(I)|}.
\]
Since $mG-2J_{\max}=G^2/(1-m)$, the value of $\mathcal T_I^{(3)}$ is exactly the term indexed by $I$ in $\Psi_{r,m,G}(J_{\max})$.

It remains to compute $\mathcal T_{I_0}^{(3)}$. By Lemma~\ref{lemexactexpansion}, after reindexing the odd neighbour variables,
\[
\mathcal T_{I_0}^{(3)}=\int_{[0,m]^r}g_3(u_1)g_3(u_r)\prod_{j=1}^{r-1}g_3(\min(u_j,u_{j+1}))\dd u_1\cdots\dd u_r.
\]
Since $Q$ is a right interval, $g_3(\min(u_j,u_{j+1}))$ is nonzero exactly when both $u_j$ and $u_{j+1}$ lie in $Q$. Therefore
\[
g_3(u_1)g_3(u_r)\prod_{j=1}^{r-1}g_3(\min(u_j,u_{j+1}))=(1-m)^{r+1}\prod_{j=1}^r\1_Q(u_j).
\]
It follows that
$
\mathcal T_{I_0}^{(3)}=(1-m)^{r+1}\rho^r=(1-m)G^r.
$
On the other hand, the final two terms of $\Psi_{r,m,G}(J_{\max})$ satisfy
\[
\frac{G^{r+1}}{m}+\frac{2(1-m)G^{r-1}}{m}J_{\max}
=\frac{G^{r+1}}{m}+\frac{(1-m)G^{r-1}}{m}\left(mG-\frac{G^2}{1-m}\right)=(1-m)G^r.
\]
Thus $\mathcal T_{I_0}^{(3)}$ is equal to the final two terms of $\Psi_{r,m,G}(J_{\max})$. Therefore $\Psi_{r,m,G}(J_{\max})=T_r^{(3)}(m,G)$.
\end{proof}

By Claims~\ref{D_2} and~\ref{D_3},
\[
t(P_{2r+1},W_D)\le\Psi_{r,m,G}(J)\le\max\{T_r^{(2)}(m,G),T_r^{(3)}(m,G)\}.
\]
The boundary cases follow by taking limits from the region $0<m<1$ and $0<G<m(1-m)$, since all quantities involved are continuous in the corresponding threshold parameters.
\end{proof}

\section{Proof of the main theorem}\label{secproof}

We first treat the interval endpoint. This is the endpoint in the collapse theorem where the quantile function has constant gap.

\begin{lemma}\label{lemtwoblockendpoint}
Let $r\ge1$, let $D=[a,b]\subseteq[0,1]$, and let $c=t(K_2,W_D)$. Then
\[
t(P_{2r+1},W_D)\le \max\{S_r(c),c^{r+1/2}\}.
\]
\end{lemma}

\begin{proof}
If $c=0$, then $b=a$, so $t(P_{2r+1},W_D)=0$. If $c=1$, then $a=0$ and $b=1$, so $W_D=1$ almost everywhere and $t(P_{2r+1},W_D)=1$. Hence the claim is immediate in these two cases. Assume $0<c<1$. Then $a<b$. Since $D=[a,b]$, by Lemma~\ref{lemthresholdedge} we have 
$
c=t(K_2,W_D)=b^2-a^2.
$
For this interval, the quantile obtained from Lemma~\ref{lemquantile} is $q(u)=a+u$ for $0\le u\le b-a$, and hence $g(u)=a$. By Lemma~\ref{lemexactexpansion},
\[
\begin{aligned}
t(P_{2r+1},W_D)
&=\sum_{I\in\cI(P_{2r+1})}\int_{[0,b-a]^{V(P_{2r+1})\setminus I}}\prod_{i\in I}a\prod_{j\in V(P_{2r+1})\setminus I}\dd u_j\\
&=\sum_{I\in\cI(P_{2r+1})}a^{|I|}(b-a)^{2r+1-|I|}=\sum_{j=0}^{r+1}\binom{2r+2-j}{j}a^j(b-a)^{2r+1-j}.
\end{aligned}
\]
Let $\rho=\frac{a}{b-a}$. Then $a=\rho(b-a)$ and $b=(1+\rho)(b-a)$. Hence
\[
c=b^2-a^2=((1+\rho)^2-\rho^2)(b-a)^2=(1+2\rho)(b-a)^2.
\]
Thus $b-a=\sqrt{\frac{c}{1+2\rho}}$, and
\[
t(P_{2r+1},W_D)=(b-a)^{2r+1}\sum_{j=0}^{r+1}\binom{2r+2-j}{j}\rho^j=c^{r+1/2}\frac{A_r(\rho)}{(1+2\rho)^{r+1/2}},
\]
where
$
A_r(\rho)=\sum_{j=0}^{r+1}\binom{2r+2-j}{j}\rho^j.
$
Let
$
H_r(\rho)=\frac{A_r(\rho)}{(1+2\rho)^{r+1/2}}.
$
Then we have $t(P_{2r+1},W_D)=c^{r+1/2}H_r(\rho)$.

We next determine where $H_r$ can be largest on the feasible range of $\rho$. Since $b=(1+\rho)\sqrt{\frac{c}{1+2\rho}}$, the condition $b\le1$ is exactly
$
(1+\rho)\sqrt{\frac{c}{1+2\rho}}\le1.
$
Notice that the function $\rho\mapsto(1+\rho)\sqrt{\frac{c}{1+2\rho}}$ is increasing on $[0,\infty)$, since
$
\frac{d}{d\rho}\log\left((1+\rho)\sqrt{\frac{c}{1+2\rho}}\right)=\frac{\rho}{(1+\rho)(1+2\rho)}\ge0.
$
At $\rho=0$, this value is $\sqrt c<1$, while it tends to infinity as $\rho\to\infty$. Hence the feasible values of $\rho$ form an interval $[0,\rho_{\max}]$, where $\rho_{\max}$ is determined by $b=1$. 

It remains to control $H_r$ on this interval. Write $\alpha_j=\binom{2r+2-j}{j}$. Then $A_r(\rho)=\sum_{j=0}^{r+1}\alpha_j\rho^j$, and
\[
H_r'(\rho)=\frac{(1+2\rho)A_r'(\rho)-(2r+1)A_r(\rho)}{(1+2\rho)^{r+3/2}}.
\]
Let $N_r(\rho)=(1+2\rho)A_r'(\rho)-(2r+1)A_r(\rho)$. The constant coefficient of $N_r$ is $\alpha_1-(2r+1)\alpha_0=0$. For $1\le j\le r$, the coefficient of $\rho^j$ in $N_r$ is
\[
(j+1)\alpha_{j+1}+(2j-2r-1)\alpha_j=-\alpha_j\frac{j(2r+1-2j)}{2r+2-j}<0.
\]
The coefficient of $\rho^{r+1}$ in $N_r$ is $\alpha_{r+1}>0$. Thus the nonzero coefficients of $N_r$ have signs $-,\ldots,-,+$. By Descartes' rule of signs, $N_r$ has at most one positive zero. Since $N_r(\rho)<0$ for all sufficiently small positive $\rho$ and $N_r(\rho)>0$ for all sufficiently large $\rho$, it has exactly one positive zero. Therefore $H_r$ decreases and then increases on $[0,\infty)$.
Therefore,
it follows that
$
H_r(\rho)\le \max\{H_r(0),H_r(\rho_{\max})\}
$
for every feasible $\rho$. Multiplying by $c^{r+1/2}$, we obtain
\[
t(P_{2r+1},W_D)\le \max\{c^{r+1/2}H_r(0),c^{r+1/2}H_r(\rho_{\max})\}.
\]
At the left endpoint $\rho=0$, one has $a=0$, $D=[0,\sqrt c]$, and $H_r(0)=1$. Thus $c^{r+1/2}H_r(0)=c^{r+1/2}$. At the right endpoint $\rho=\rho_{\max}$, one has $b=1$. Therefore $D=[a,1]$ and $c=1-a^2$, so $a=\sqrt{1-c}$. Hence
\[
c^{r+1/2}H_r(\rho_{\max})=\sum_{j=0}^{r+1}\binom{2r+2-j}{j}(1-a)^{2r+1-j}a^j=S_r(c).
\]
Consequently,
$
t(P_{2r+1},W_D)\le \max\{S_r(c),c^{r+1/2}\}
$ and 
this proves the lemma.
\end{proof}

The same calculation identifies the transition point introduced in the introduction.

\begin{lemma}\label{lemtransition}
The polynomial $E_r(w)/w$ has a unique positive zero $w_r$. If $c_r=(1+2w_r)/(1+w_r)^2$, then $S_r(c)=c^{r+1/2}$ holds in $(0,1)$ exactly at $c=c_r$. Moreover $S_r(c)>c^{r+1/2}$ for $0<c<c_r$, while $S_r(c)<c^{r+1/2}$ for $c_r<c<1$.
\end{lemma}

\begin{proof}
For the quasi-star graphon, write $w=a/s$. Then $w>0$, $c=(1+2w)/(1+w)^2$, and $S_r(c)=c^{r+1/2}H_r(w)$ with $H_r$ as in the proof of Lemma~\ref{lemtwoblockendpoint}. The equation $S_r(c)=c^{r+1/2}$ is equivalent to $H_r(w)=1$, which is equivalent to $A_r(w)^2=(1+2w)^{2r+1}$. This is the equation $E_r(w)=0$. Since $E_r(0)=0$, the positive solutions are exactly the positive zeros of $E_r(w)/w$.

The proof of Lemma~\ref{lemtwoblockendpoint} shows that $H_r$ decreases and then increases on $[0,\infty)$. Also $H_r(0)=1$, $H_r(w)<1$ for all sufficiently small positive $w$, and $H_r(w)\to\infty$ as $w\to\infty$. Hence $H_r(w)=1$ has exactly one positive solution. This proves that $E_r(w)/w$ has a unique positive zero $w_r$.

The map $w\mapsto (1+2w)/(1+w)^2$ is strictly decreasing on $(0,\infty)$. Therefore the unique positive solution $w_r$ gives a unique transition value $c_r=(1+2w_r)/(1+w_r)^2$. Since $H_r(w)<1$ for $0<w<w_r$ and $H_r(w)>1$ for $w>w_r$, and since larger $w$ corresponds to smaller $c$, we get $S_r(c)>c^{r+1/2}$ for $0<c<c_r$ and $S_r(c)<c^{r+1/2}$ for $c_r<c<1$.
\end{proof}

We next prepare the three-step endpoint. The following polynomial is the word model for the independent-set expansion of a three-step threshold graphon.

For $n\ge0$, let $\mathcal{W}_n$ be the set of words of length $n$ over the alphabet $\{a,b,\ell\}$ in which every occurrence of $b$ has only $\ell$ as a neighbour when the neighbour exists. Equivalently, the adjacent pairs $ab$, $ba$, and $bb$ are forbidden. Let
\[
F_n(a,b,\ell)=\sum_{w\in\mathcal{W}_n}\prod_{i=1}^n w_i.
\]
For $0\le x\le y\le1$, let $T_r(x,y)$ be the $P_{2r+1}$ density of the three-step threshold graphon associated with $D=[0,x]\cup[y,1]$. This graphon has edge density $1-y^2+x^2$, and
\[
T_r(x,y)=F_{2r+1}(x,y-x,1-y).
\]

The next counting lemma gives the coefficients of $F_n$ explicitly.

\begin{lemma}\label{lemwordcoefficients}
Let $f_{\alpha,\beta,\gamma}$ be the coefficient of $a^\alpha b^\beta\ell^\gamma$ in $F_n$, where $\alpha+\beta+\gamma=n$. If $0\le\beta\le\gamma$, then
\[
f_{\alpha,\beta,\gamma}=\binom{\gamma+1}{\beta}\binom{\alpha+\gamma-\beta}{\alpha}.
\]
If $\beta=\gamma+1$, then $f_{0,\gamma+1,\gamma}=1$ and $f_{\alpha,\gamma+1,\gamma}=0$ for $\alpha>0$. In all other cases $f_{\alpha,\beta,\gamma}=0$.
\end{lemma}

\begin{proof}
Fix the positions occupied by the $\gamma$ letters $\ell$. They create $\gamma+1$ gaps. A gap may contain any number of letters $a$, or it may contain a single letter $b$, but it cannot contain both types and it cannot contain more than one letter $b$. If $0\le\beta\le\gamma$, choose the $\beta$ gaps containing a letter $b$, and then distribute the $\alpha$ letters $a$ among the remaining $\gamma+1-\beta$ gaps. This gives $\binom{\gamma+1}{\beta}\binom{\alpha+\gamma-\beta}{\alpha}$. If $\beta=\gamma+1$, every gap contains a letter $b$, so this is possible only when $\alpha=0$, and then there is exactly one word. The remaining cases are impossible.
\end{proof}

The following identity is the algebraic form of removing the two forced neighbours of a marked middle-block vertex. Here $\partial_a,\partial_b,\partial_\ell$ denote formal partial differentiation with respect to the variables $a,b,\ell$. The operator $\mathcal L$ is the directional derivative that appears when the three-step parameters $a=x$, $b=y-x$, and $\ell=1-y$ vary along a fixed edge-density curve.

\begin{lemma}\label{lemfirstderivativeidentity}
Let
\[
\mathcal{L}=(a+b)\partial_a-b\partial_b-a\partial_\ell,\qquad \cE=a\partial_a-b\partial_b.
\]
For every $n\ge2$, we have
\[
\mathcal{L} F_n=ab\,\cE F_{n-2}.
\]
\end{lemma}

\begin{proof}
We compare coefficients. Let $f_{\alpha,\beta,\gamma}$ be as in Lemma~\ref{lemwordcoefficients}, with the convention that it is zero if one of the indices is negative. The coefficient of $a^\alpha b^\beta\ell^\gamma$ in $\mathcal{L} F_n$ is
\[
(\alpha-\beta)f_{\alpha,\beta,\gamma}+(\alpha+1)f_{\alpha+1,\beta-1,\gamma}-(\gamma+1)f_{\alpha-1,\beta,\gamma+1}.
\]
The coefficient of the same monomial in $ab\,\cE F_{n-2}$ is $(\alpha-\beta)f_{\alpha-1,\beta-1,\gamma}$. Substituting the formula of Lemma~\ref{lemwordcoefficients} and using Pascal's identity, we obtain the equality in the ranges $0\le\beta\le\gamma$ and $\beta=\gamma+1$. It remains to check the boundary case $\beta=\gamma+2$. If $\alpha\ne1$, then both sides vanish. If $\alpha=1$, then the left hand side is $-(\gamma+1)f_{0,\gamma+2,\gamma+1}=-(\gamma+1)$, while the right hand side is $(1-\gamma-2)f_{0,\gamma+1,\gamma}=-(\gamma+1)$. In all remaining cases, both sides vanish. Hence the two polynomials are equal.
\end{proof}

The next lemma is the combinatorial positivity statement behind the Bernstein certificate.

\begin{lemma}\label{lempositivecoefficients}
For every $m\ge1$, the polynomial $\mathcal{L}\cE F_m$ has nonnegative integer coefficients.
\end{lemma}

\begin{proof}
Let $p_{\alpha,\beta,\gamma}$ be the coefficient of $a^\alpha b^\beta\ell^\gamma$ in $\mathcal{L}\cE F_m$, where $\alpha+\beta+\gamma=m$. By applying $\mathcal{L}$ to $\cE F_m$, we obtain that
\[
\begin{aligned}
p_{\alpha,\beta,\gamma}={}&(\alpha-\beta)^2f_{\alpha,\beta,\gamma}\\
&+(\alpha+1)(\alpha-\beta+2)f_{\alpha+1,\beta-1,\gamma}\\
&+(\gamma+1)(\beta-\alpha+1)f_{\alpha-1,\beta,\gamma+1}.
\end{aligned}
\]
Here the coefficients $f_{\alpha,\beta,\gamma}$ are those of $F_m$, and the same zero convention is used as above.

Assume first that $0\le\beta\le\gamma$. Let $k=\gamma-\beta$. Substituting the formula from Lemma~\ref{lemwordcoefficients} and simplifying gives
\[
p_{\alpha,\beta,\gamma}=f_{\alpha,\beta,\gamma}\frac{\alpha\beta(\alpha-\beta)^2+\alpha\beta(2\alpha+\beta+6k+9)+(\alpha+2\beta)(k+1)(k+2)}{(k+1)(k+2)}.
\]
The denominator is positive and every term in the numerator is nonnegative, so $p_{\alpha,\beta,\gamma}\ge0$.

Assume next that $\beta=\gamma+1$. If $\alpha=0$, then $p_{0,\gamma+1,\gamma}=2(\gamma+1)\ge0$. If $\alpha\ge1$, then
\[
p_{\alpha,\gamma+1,\gamma}=(\gamma+1)\left((\alpha-\gamma)^2+3\gamma+5\right)\ge0.
\]
Finally assume that $\beta=\gamma+2$. The only possible nonzero coefficient occurs when $\alpha=1$, and then $p_{1,\gamma+2,\gamma}=(\gamma+1)(\gamma+2)\ge0$. All remaining coefficients are zero. Therefore every coefficient of $\mathcal{L}\cE F_m$ is nonnegative.
\end{proof}

We now translate the previous lemma into the Bernstein form used to control the three-step endpoint.

\begin{lemma}\label{lembernsteinpositive}
Let $m=2r-1$ and write
$
\mathcal{L}\cE F_m(a,b,\ell)=\sum_{\alpha+\beta+\gamma=m}p_{\alpha,\beta,\gamma}a^\alpha b^\beta\ell^\gamma.
$
Then
\[
\frac1y\mathcal{L}\cE F_m(uy,(1-u)y,1-y)=\sum_{i=0}^m\sum_{j=0}^{m-1}C_{ij}^{(r)}u^i(1-u)^{m-i}y^j(1-y)^{m-1-j},
\]
where
\[
C_{ij}^{(r)}=\sum_\alpha p_{\alpha,j+1-\alpha,m-1-j}\binom{m-1-j}{i-\alpha}.
\]
Here $\binom{N}{K}=0$ if $K<0$ or $K>N$. In particular, all $C_{ij}^{(r)}$ are nonnegative.
\end{lemma}

\begin{proof}
By Lemma~\ref{lempositivecoefficients}, all $p_{\alpha,\beta,\gamma}$ are nonnegative. Also $p_{0,0,m}=0$, so division by $y$ creates no negative power. After substituting $a=uy$, $b=(1-u)y$, and $\ell=1-y$, the monomial $a^\alpha b^\beta\ell^\gamma/y$ becomes $u^\alpha(1-u)^\beta y^{\alpha+\beta-1}(1-y)^\gamma$. Since $\alpha+\beta+\gamma=m$, the exponent of $y$ is $m-1-\gamma$. Multiplying by $1=(u+(1-u))^\gamma$,
\[
u^\alpha(1-u)^\beta=\sum_{s=0}^\gamma\binom{\gamma}{s}u^{\alpha+s}(1-u)^{\beta+\gamma-s}.
\]
For fixed $i$ and $j$, the coefficient of $u^i(1-u)^{m-i}y^j(1-y)^{m-1-j}$ is \[ C_{ij}^{(r)}=\sum_\alpha p_{\alpha,j+1-\alpha,m-1-j}\binom{m-1-j}{i-\alpha}. \] Each summand in this expression is nonnegative.
\end{proof}

The next lemma proves the endpoint comparison for the three-step family. This is the point at which the combinatorial Bernstein certificate enters the proof.

\begin{lemma}\label{lemthreestependpoint}
Let $r\ge2$. For every $0\le x\le y\le1$, one has
\[
T_r(x,y)\le \max\{S_r(1-y^2+x^2),(1-y^2+x^2)^{r+1/2}\}.
\]
\end{lemma}

\begin{proof}
Let $c=1-y^2+x^2$. The cases $c=0$ and $c=1$ are immediate. If $c=0$, then $x=0$ and $y=1$, so $T_r(x,y)=0$. If $c=1$, then $x=y$, and the corresponding three-step graphon is complete up to a null set, so $T_r(x,y)=1$. Assume now that $0<c<1$. For fixed $c$, the admissible points are exactly the points on the curve $y=\sqrt{1-c+x^2}$ with $0\le x\le\sqrt c$. It is enough to prove that $T_r(x,\sqrt{1-c+x^2})$ attains its maximum on this interval at one of the two endpoints.

Let $a=x$, $b=y-x$, and $\ell=1-y$. On the open interval $0<x<\sqrt c$, one has $a>0$, $b>0$, $\ell>0$, and $a+b=y>0$. Along the curve,
\[
\frac d{dx}=\partial_x+\frac xy\partial_y=\frac1{a+b}\mathcal{L}.
\]
Since $T_r(x,y)=F_{2r+1}(a,b,\ell)$ and $m=2r-1$, Lemma~\ref{lemfirstderivativeidentity} gives
\[
\frac d{dx}T_r(x,y)=\frac1{a+b}\mathcal{L} F_{m+2}(a,b,\ell)=\frac{ab}{a+b}\cE F_m(a,b,\ell).
\]
Let $P_r(x,y)=\cE F_m(x,y-x,1-y)$. The preceding identity shows that the sign of the derivative of $T_r$ is the sign of $P_r$ in the interior of the interval, since $ab/(a+b)>0$. The derivative of $P_r$ along the same curve is
\[
\frac d{dx}P_r(x,y)=\frac1{a+b}\mathcal{L}\cE F_m(a,b,\ell).
\]
Since $a+b=y$, write $u=x/y$. Then $0<u<1$, $0<y<1$, and
\[
a=uy,\qquad b=(1-u)y,\qquad \ell=1-y.
\]
By Lemma~\ref{lembernsteinpositive},
\[
\frac d{dx}P_r(x,y)=\frac1y\mathcal{L}\cE F_m(uy,(1-u)y,1-y)
=\sum_{i=0}^m\sum_{j=0}^{m-1}C_{ij}^{(r)}u^i(1-u)^{m-i}y^j(1-y)^{m-1-j}.
\]
All coefficients $C_{ij}^{(r)}$ are nonnegative, and all basis terms are nonnegative on the interval. Hence $\frac d{dx}P_r(x,y)\ge0$. Therefore $P_r(x,y)$ is nondecreasing along the curve.

At the right endpoint one has $x=\sqrt c$, $y=1$, and $\ell=0$. Since $m\ge3$, the only word contributing to $F_m(a,b,0)$ is $a^m$, so $P_r(\sqrt c,1)=m c^{m/2}>0$. Therefore $P_r$ is either nonnegative throughout the interval or changes sign once from negative to positive. Hence $T_r(x,\sqrt{1-c+x^2})$ is either nondecreasing or decreases and then increases. Its maximum is attained at an endpoint.

At the left endpoint $x=0$ and $y=\sqrt{1-c}$, the value is
\[
T_r(0,\sqrt{1-c})=\sum_{j=0}^{r+1}\binom{2r+2-j}{j}(1-\sqrt{1-c})^{2r+1-j}(\sqrt{1-c})^j=S_r(c).
\]
At the right endpoint $x=\sqrt c$ and $y=1$, the value is $T_r(\sqrt c,1)=c^{r+1/2}$. Thus
\[
T_r(x,y)\le \max\{S_r(c),c^{r+1/2}\}.
\]
This proves the lemma.
\end{proof}

The case $r=1$ was settled by Ahlswede and Katona. We record it in the same normalization.

\begin{lemma}[Ahlswede and Katona {\cite{AhlswedeKatona1978}}]\label{lempthree}
For every $0\le c\le1$, one has
\[
M_1(c)=\max\{S_1(c),c^{3/2}\}.
\]
\end{lemma}

We now prove the graphon form of the theorem.

\begin{proof}[Proof of Theorem~\ref{thmgraphonform}]
The case $r=1$ is Lemma~\ref{lempthree}. Assume $r\ge2$. The lower bound follows from the quasi-star and quasi-clique graphons. We prove the upper bound.

The cases $c=0$ and $c=1$ are immediate. If $c=0$, then every admissible graphon has edge density zero, and hence has zero $P_{2r+1}$ density. If $c=1$, then every graphon has $P_{2r+1}$ density at most $1$, and the complete graphon attains this value. Assume now that $0<c<1$.

By Lemma~\ref{lemthresholdreduction}, it is enough to consider threshold graphons. It is also enough to consider threshold graphons of edge density exactly $c$. If $2\int_Dt\dd t<c$, then the non-atomic measure $2t\dd t$ allows us to enlarge $D$ to a measurable set $D'$ with $2\int_{D'}t\dd t=c$. Since $W_D\le W_{D'}$ pointwise, one has $t(P_{2r+1},W_D)\le t(P_{2r+1},W_{D'})$.

Let $W_D$ have edge density $c$. Let $m=\Leb(D)$, and let $q(u)=u+g(u)$ be the quantile representation obtained from Lemma~\ref{lemquantile}. Set $G=\int_0^m g(u)\dd u$. By Lemma~\ref{lemedgeconstraint}, one has $c=t(K_2,W_D)=m^2+2G$. Since $0\le g\le1-m$, we have $0\le G\le m(1-m)$. Moreover, if $m=0$, then $G=0$ and hence $c=0$; if $m=1$, then $g=0$ almost everywhere and hence $c=1$. Therefore, we have $0<m<1$ since $0<c<1$.

If $G=0$, then $g=0$ almost everywhere. Hence $q(u)=u$ almost everywhere, and $D$ agrees up to a null set with $[0,m]$. The graphon is the quasi-clique graphon of edge density $m^2=c$, so
\[
t(P_{2r+1},W_D)=c^{r+1/2}\le \max\{S_r(c),c^{r+1/2}\}.
\]
If $G=m(1-m)$, then $g=1-m$ almost everywhere. Hence $q(u)=u+1-m$ almost everywhere, and $D$ agrees up to a null set with $[1-m,1]$. The graphon is the quasi-star graphon of edge density $c=m^2+2m(1-m)$, so
\[
t(P_{2r+1},W_D)=S_r(c)\le \max\{S_r(c),c^{r+1/2}\}.
\]

It remains to treat the case $0<G<m(1-m)$. By Theorem~\ref{thmonedefectcollapse},
\[
t(P_{2r+1},W_D)\le\max\{T_r^{(2)}(m,G),T_r^{(3)}(m,G)\}.
\]
The two-block endpoint satisfies
\[
T_r^{(2)}(m,G)\le\max\{S_r(c),c^{r+1/2}\}
\]
by Lemma~\ref{lemtwoblockendpoint}. The three-step endpoint satisfies
\[
T_r^{(3)}(m,G)\le\max\{S_r(c),c^{r+1/2}\}
\]
by Lemma~\ref{lemthreestependpoint}. Hence $t(P_{2r+1},W_D)\le\max\{S_r(c),c^{r+1/2}\}$ for every threshold graphon of edge density $c$. This proves $M_r(c)=\max\{S_r(c),c^{r+1/2}\}.$
\end{proof}

The finite graph statement follows from the graphon statement by the standard passage from homomorphism densities to injective copies.

\begin{proof}[Proof of Theorem~\ref{thmmain}]
Let $W_G$ be the graphon associated with $G$. Then $t(K_2,W_G)=2e(G)/n^2=c$, and by Theorem~\ref{thmgraphonform}, $t(P_{2r+1},W_G)\le\max\{S_r(c),c^{r+1/2}\}$. Also $t(P_{2r+1},W_G)=\hom(P_{2r+1},G)/n^{2r+1}$. The number of homomorphisms from $P_{2r+1}$ to $G$ that identify two vertices of $P_{2r+1}$ is $O(n^{2r})$. Every simple unlabelled copy of $P_{2r+1}$ is counted by exactly two injective homomorphisms. Hence
\[
N(P_{2r+1},G)=\frac12t(P_{2r+1},W_G)n^{2r+1}+O(n^{2r})\le\frac12\max\{S_r(c),c^{r+1/2}\}n^{2r+1}+O(n^{2r}).
\]
By Lemma~\ref{lemtransition}, the maximum is $S_r(c)$ for $0<c\le c_r$ and $c^{r+1/2}$ for $c_r\le c<1$. The quasi-star and quasi-clique graph sequences converge to the corresponding graphons, so the upper bounds are asymptotically sharp.
\end{proof}

\section{Concluding remarks}\label{secconclusion}

This paper determines the asymptotic maximum number of copies of all paths with an odd number of vertices under a fixed edge density. The results show that for each path $P_{2r+1}$, the extremal graphon is obtained by a quasi-star or a quasi-clique, and the transition between the two is given by an explicit algebraic point. Paths with an even number of vertices belong to a simpler degenerate case. To see this, let $P_{2r}$ be the path on $2r$ vertices. Then $P_{2r}$ contains a perfect matching of size $r$. Therefore, for any graphon $W$, since $0\le W\le1$, we have $t(P_{2r},W)\le t(K_2,W)^r$. Thus $M_{P_{2r}}(c)=c^r$, and this value is attained by the quasi-clique graphon. In other words, paths with an even number of vertices are quasi-clique-only cases, while the paths $P_{2r+1}$ proved in this paper are precisely the cases where a genuine quasi-star versus quasi-clique phase transition occurs.

Nagy proposed a broader quasi-star/quasi-clique conjecture: whether for every fixed graph $H$ and every edge density $c$, the asymptotic extremum under fixed edge density is always given by a quasi-star or a quasi-clique \cite{Nagy2017}. This conjecture was later disproved by Day and Sarkar \cite{DaySarkar2021}. Blekherman and Patel subsequently provided a unified framework from the perspective of threshold graphons, showing that for certain graphs and certain edge densities, the optimal threshold graphon cannot degenerate into a quasi-star or a quasi-clique but requires a more complex threshold structure \cite{BlekhermanPatel2024}.

Therefore, the correct question is not whether quasi-star and quasi-clique are always sufficient, but rather to identify which families of graphs still satisfy this two-construction principle. There are already several important positive examples. For a graph $H$ with no isolated vertices, the fractional independence number of $H$ is defined as
\[ \alpha^*(H)=\max\left\{\sum_{v\in V(H)}x_v\ \middle|\ x_u+x_v\le1\ \text{for every }uv\in E(H),\ x_v\ge0\ \text{for every }v\in V(H)\right\}. \] If a graph $H$ with no isolated vertices satisfies $\alpha^*(H)=|V(H)|/2$, then $H$ is usually called balanced; such graphs enjoy a stronger quasi-clique-only conclusion~\cite{Alon1981,Alon1986}, namely that the quasi-clique gives the asymptotic extremum at all edge densities. Stars provide a genuine positive family exhibiting the quasi-star/quasi-clique dichotomy \cite{ReiherWagner2018,CairncrossMubayi2025}. The results of this paper show that all odd paths also belong to this positive family with the two-phase behavior.

\begin{problem}\label{probcharacterize}
Determine all finite graphs $H$ for which the following holds for every $0\le c\le1:$
\[
M_H(c)=\max\{t(H,A_{\st}^c),t(H,A_{\clq}^c)\}.
\]
\end{problem}

Beyond this, one may further investigate the corresponding stability and uniqueness questions, which concern the structure of all asymptotically extremal graph sequences for graphs $H$ satisfying Problem~\ref{probcharacterize}. One would like to know whether, apart from possible transition points, every such sequence must be close in cut distance to either the quasi-star graphon or the quasi-clique graphon, and whether at the transition point these two are the only extremal structures or whether there exist other asymptotically extremal structures given by mixtures of the two or by more complicated threshold graphons.

The answers to these questions may depend on the matching structure of $H$, the fractional independence number, and whether, after Blekherman and Patel's threshold reduction, threshold graphons with three or more steps are required. The proof in this paper shows that excluding general threshold graphons can require additional structural collapse and an algebraic certificate, even for the family of paths. Therefore, Problem~\ref{probcharacterize} may require a combination of the combinatorial structure of $H$ and a multi-parameter analysis of the corresponding threshold graphon variational problem.

\bibliographystyle{plain}
\bibliography{refs}
\end{document}